\documentclass{article}
\usepackage{amssymb,amsmath,amsthm,amscd,verbatim,graphicx}
\usepackage[english]{babel}
\usepackage[latin1]{inputenc}
\usepackage{times, mathrsfs}
\usepackage[T1]{fontenc}
\usepackage{enumerate}
\usepackage{hyperref}
\usepackage{amsrefs}

\newtheorem{thm}{Theorem}[section]

\newtheorem{cor}[thm]{Corollary}
\newtheorem{lem}[thm]{Lemma}
\newtheorem{que}[thm]{Question}

\theoremstyle{definition}
\newtheorem{df}[thm]{Definition}

\theoremstyle{remark}
\newtheorem{rem}[thm]{Remark}

\numberwithin{equation}{section}

\newcommand{\N}{\mathbb{N}}

\newcommand{\mf}{\mathfrak}

\renewcommand{\to}{\rightarrow}

\newcommand{\restrict}{\upharpoonright}

\newcommand{\E}{\mathbb{E}}

\newcommand{\R}{\mathbb{R}}
\newcommand{\Z}{\mathbb{Z}}
\renewcommand{\P}{\mathbb{P}}
\newcommand{\eps}{\varepsilon}
\newcommand{\nil}{\varnothing}
\newcommand{\Pb}{\mathbb{P} \left\{ }

\newcommand{\rb}{\right\} }
\newcommand{\Var}{\mathbf{Var}}
\DeclareMathOperator{\osc}{osc}
\begin{document}

\title{Kolmogorov complexity and strong approximation of Brownian motion}

\author{Bj\o rn Kjos-Hanssen\footnote{This material is based upon work supported by the National Science 
Foundation under Grants No.\ 0652669 and 0901020. Thanks are due to the anonymous referee for very helpful comments and to Jacob Woolcutt for assistance with the production of Figure \ref{fig:Fig}.}\\
Tam\'as Szabados}
\maketitle

\begin{abstract}
Brownian motion and scaled and interpolated simple random walk can be jointly embedded in a probability space in such a way that almost surely the $n$-step walk is within a uniform distance $O(n^{-1/2}\log n)$ of the Brownian path for all but finitely many positive integers $n$. Almost surely this $n$-step walk will be incompressible in the sense of Kolmogorov complexity, and all {Martin-L\"of random} paths of Brownian motion have such an incompressible close approximant. This strengthens a result of Asarin, who obtained the bound $O(n^{-1/6} \log n)$. The result cannot be improved to $o(n^{-1/2}{\sqrt{\log n}})$.
\end{abstract}

\section{Introduction}\label{KobeBerkeley}

\subsection{Algorithmic randomness and probability}

Almost sure statements in probability theory usually do not come with examples, but \emph{algorithmically random} objects, defined in terms of Turing computability theory, have most of the properties expected in almost sure behavior. The value of their study lies in the extent to which all almost sure properties of interest are reflected in each algorithmically random object. 

Cara\-th\'eodory's measure algebra isomorphism theorem gives a sense in which it suffices to consider the case of infinite binary sequences with fair-coin measure, equivalently the unit interval $[0,1]$ with Lebesgue measure \cite{KN}. This case has been deeply studied; see for instance the recent book by Nies \cite{NiesBook}. In the present paper we follow up on work of Asarin and Pokrovskii \cite{AP} and Fouch\'e \cite{F} on algorithmic randomness in the context of \emph{Brownian motion}, i.e., Wiener measure on the space of continuous functions $C[0,1]$. Our goal is to strengthen and ``explain'' a theorem of Asarin relating Brownian motion and random walks to Kolmogorov complexity. The isomorphism between the measure algebras of $C[0,1]$ and $[0,1]$ does not help here, because the theorem involves metric structure. 

\subsection{Schnorr and Asarin's theorems}

The algorithmic randomness of an infinite object $A$ may often be expressed in terms of the complexity of its finite approximations. This is most well known in the case of infinite binary strings $A\in 2^\omega=\{0,1\}^\infty$.

\begin{df}\label{1}
 Let $\mu$ be the fair-coin measure on $2^\omega$. A real $A\in 2^\omega$ is \emph{Martin-L\"of random} if for each uniformly $\Sigma^0_1$ sequence $\{U_n\}_{n\in\N}$, $U_{n}\subseteq 2^{\omega}$, with $\mu U_n\le 2^{-n}$, we have $A\not\in\bigcap_n U_n$.
\end{df}

Let $K$ denote prefix-free Kolmogorov complexity \cite{LV} over either the usual alphabet $\{0,1\}$ or the alternative $\{\pm 1\}=\{1,-1\}$. The Kolmogorov complexity of an element of $\{\pm 1\}^{n}$ can be identified with the Kolmogorov complexity of the corresponding string in $\{0,1\}^{n}$ under (say) the map given by $-1\mapsto 0$ and $1\mapsto 1$.

\begin{thm}[Schnorr]\label{Sch}
$A\in 2^\omega$ is Martin-L\"of random if and only if there is a constant $c$ such that for all $n$, $K(A\restrict n)\ge n-c$.
\end{thm}

In his dissertation \cite{AsarinThesis}, Kolmogorov's student Asarin defined Martin-L\"of random \emph{Brownian motion} analogously to Definition \ref{1} and obtained an analogue of Schnorr's Theorem \ref{Sch} which we now describe.\footnote{Theorem \ref{Sch} was announced by Chaitin \cite{Chaitin}, and attributed to Schnorr (who was the referee of the paper) without proof. The first published proof (in a form generalized to arbitrary computable measures) appeared in the work of G\'acs \cite{Gacs}.}

Let $C[0,1]$ denote the space of continuous functions $f:[0,1]\to\R$ with the uniform metric $d$ given by
\[
	d(f,g)=\max_{t\in [0,1]} |f(x)-g(x)|
\]
and with the Wiener measure $\P$ underlying Brownian motion as in Durrett \cite{Durrett}. Let $\mf S$ be the set of all balls contained in $C$ with rational radii whose centers are piecewise linear functions with rational break points and rational values at the break points. The set $\mf S$ is countable, and we can specify a member of $\mf S$ by specifying a finite list of rational numbers. We may write $\mf S=\{T_n: n\in\N\}$, where the list of rational numbers representing $T_n$ is uniformly computable.
For any total computable function $\psi:\N\times\N\to\N$, the set $U_n=\bigcup_m T_{\psi(n,m)}$ is called a $\Sigma^0_1$ subset of $C[0,1]$. A set $M\subseteq C[0,1]$ is a \emph{Martin-L\"of null set} if there is a total computable function $\psi:\N\times\N\to\N$ such that for $U_n=\bigcup_m T_{\psi(n,m)}$, we have $\P(U_n)\le 2^{-n}$ and $M\subseteq\bigcap_n U_n$.

\begin{thm}[Asarin \cite{AP}]
The union $M_0$ of all Martin-L\"of null sets is a Martin-L\"of null set.
\end{thm}

A function $x\in C[0,1]$ is called \emph{Martin-L\"of random} if $x\notin M_0$. Next, we define the analogue of the other side of Schnorr's Theorem. A sequence of binary strings $\alpha_n\in\{\pm 1\}^n$, $n\in\N$, is said to be \emph{complex} if there is a constant $c>0$ such that for all $n\in\N$, 
\[
	K(\alpha_n)\ge n-c.
\]
Let $C_n$ be the set of all functions $f\in C[0,1]$ that are linear with slope $\pm\sqrt{n}$ in each interval $[\frac{i-1}{n},\frac{i}{n}]$, $1\le i\le n$, and such that $f(0)=0$. Each function $f\in C_n$ is associated with a finite string, where $-1$ ($1$) represents an interval where $f$ is decreasing (increasing). If this string is $x$, then we write $f=\ell_x$. Thus we can speak of a sequence of functions $f_n\in C_n$, $n\in\N$, being complex as well. We are now in a position to state Asarin's result.

\begin{thm}[Asarin \cite{AP}]\label{Asar}
$W\in C[0,1]$ is a Martin-L\"of random Brownian motion if and only if there is a constant $c$ such that for all but finitely many $n\in\N$ there is a string $x=(x_1,\ldots,x_n)\in\{\pm 1\}^n$ such that 
\[
	d(\ell_x,W)\le n^{-1/10}\qquad\text{and}\qquad K(x_1,\ldots,x_n)\ge n-c.
\]
\end{thm}
\begin{figure}[htb!]
\centering
\includegraphics[height=7.5cm]{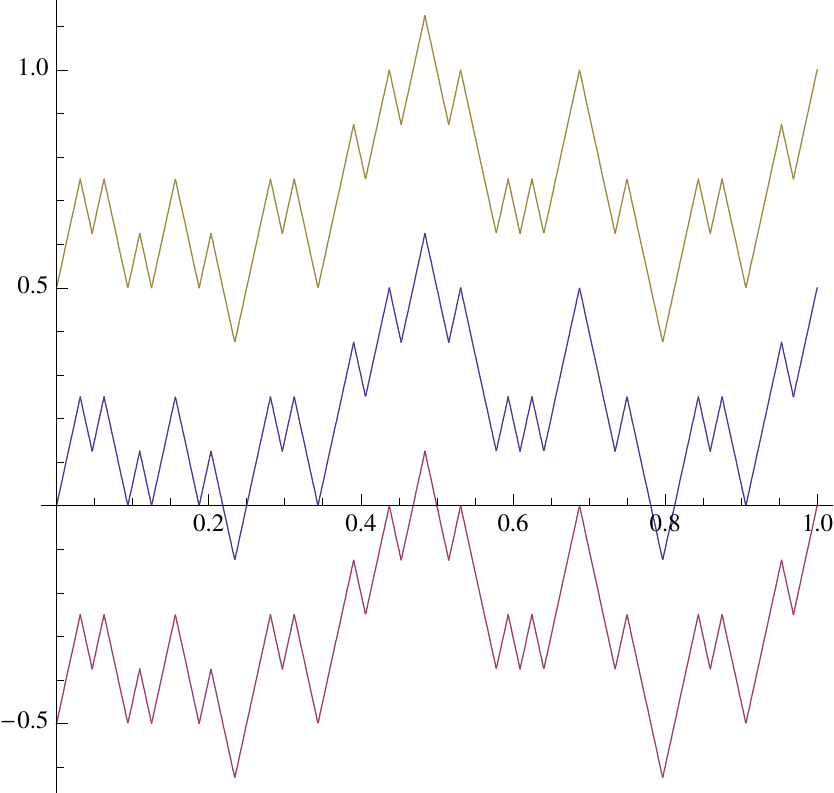}
\caption{The ball of radius $n^{-1/6}$ around $\ell_x$ for a randomly chosen $x\in \{\pm 1\}^{n}$, where $n=4^3$.}
\label{fig:Fig}
\end{figure}
The bound $n^{-1/10}$ was also quoted and used by Fouch\'e \cites{F, F1}. In his dissertation \cite{AsarinThesis}, Asarin actually improved the bound to $20 \,n^{-1/6} \log n$; an instance of the bound $n^{-1/6}$ is illustrated in Figure \ref{fig:Fig}. In the present paper we will further improve and explain these bounds. In order to do so, we will use \emph{strong approximation} of Brownian motion, i.e. the joint realization of Brownian motion and random walks on the same probability space in such a way that the Brownian path and a scaled and interpolated random walk with $n$ steps in the unit interval are almost surely within a $d$-distance $O(n^{-1/2}\log n)$. The proof of this fact is based on a finite horizon version of a celebrated embedding theorem by Koml\'os, Major, and Tusn\'ady \cites{KMT75, KMT76}, Theorem \ref{Szab}. 

\section{Joint embedding of random variables}

We start with $\{X_i\}_{i\in\N}$, a sequence of independent random variables satisfying 
\[
	\mathbb P\{X_i=1\}=\mathbb P\{X_i=-1\}=\frac{1}{2}.
\]
For any $n\ge 0$, let $S_n=\sum_{i=1}^n X_i$; the sequence of random variables $\{S_n\}_{n\in\N}$ is a \emph{random walk on $\Z$}. The \emph{piecewise linearization} $\ell_X(t)$ of $X=(X_1,\ldots,X_n)$ is a piecewise linear function with breakpoints $k/n$, such that
\[
\ell_X \left(\frac{k}{n}\right) = \frac{S(k)}{\sqrt{n}} \qquad (1 \le k \le n).
\]
For a particular value $x\in \{\pm 1\}^n$ of $X$, this agrees with our definition of $\ell_x$ in Section \ref{KobeBerkeley}. Standard \emph{Brownian motion} is a random function $W\in C[0,1]$. A good introductory reference is Durrett \cite{Durrett}. There are several results to the effect that $\ell_X$ is an approximation of $W$, similar to how a finite binary sequence $A\restrict n=(A(0),\ldots,A(n-1))$ is an approximation of $A=(A(0),A(1),\ldots)$.

Suppose that one wants to define an i.i.d. sequence $X_1, X_2, \dots$ of random variables with a given distribution so that the partial sums are as close to Brownian motion as possible.  The result we need from Koml\'os, Major, and Tusn\'ady \cites{KMT75, KMT76} is as follows.

\begin{thm}\label{labeledKMT}
Assume that $\E (X_k)=0$, $\Var (X_k)=1$ and the moment generating function $\E \left( e^{u X_k} \right) < \infty$ for $|u| \le u_0, \ u_0 > 0$. Let $S(k) = X_1+ \cdots +X_k$, $k \ge 1$, be the partial sums. If Brownian motion $\left(W(t)\right)_{t \ge 0}$ is given, then for any $n \ge 1$ there exists a sequence of transformations applied to $W(1), W(2), \dots $ so that one obtains the desired partial sums $S(1), S(2), \dots $ and the difference between the two sequences is the smallest possible:
\begin{equation}
\Pb \max_{0 \le k \le n} |S(k) - W(k)| > C_0 \log n + x \rb < K_0 e^{-\lambda x}, \label{eq:KMT0}
\end{equation}
for any $n \ge 1$ and $x>0$, where $C_0, K_0, \lambda$ are positive constants that may depend on the distribution of $X_k$, but not on $n$ or $x$. Moreover, $\lambda$ can be made arbitrarily large by choosing a large enough $C_0$.
\end{thm}

Taking $x=C_0 \log n$ in Theorem \ref{labeledKMT} one obtains
\begin{equation}
\Pb \max_{0 \le k \le n} |S(k) - W(k)| > 2 C_0 \log n \rb < K_0 n^{-\lambda C_0}, \label{eq:KMT}
\end{equation}
where $n \ge 1$ is arbitrary and one may assume that $\alpha := \lambda C_0 \ge 2$.

\begin{lem}[Brownian scaling relation, \cite{Durrett}*{p. 372}] \label{scale}
Let $\alpha\ge 0$. For standard Brownian motion (with $W_0=0$), the stochastic process
\[
	\{W_{s\alpha}\}_{s\ge 0}
\]
has the same distribution as
\[
	\{\sqrt{\alpha} \cdot W_s\}_{s\ge 0}.
\]
\end{lem}

\begin{thm}\label{Szab}
There exists a joint distribution of Brownian motion $\left(W(t)\right)_{t \ge 0}$, and for each $n \ge 1$, random variables
$X_{k,n}$ for $1 \le k \le n$, such that for each $n \ge
1$,
\begin{enumerate}
\item[1.] the random variables $X_{1,n}, \dots , X_{n,n}$ are mutually independent,

\item[2.] $\Pb X_{k,n} = 1 \rb = \Pb X_{k,n} = -1 \rb = \frac12$ ,

\item[3.] there are constants $c_1$, $c_2$, and $\alpha \ge 2$, such that if $X^{(n)} = (X_{1,n}, \dots ,
X_{n,n})$,
\[
\Pb \sup_{0\le t \le 1} |W(t) - \ell_{X^{(n)}} (t)| \ge c_1 \frac{\log n}{\sqrt{n}} \rb \le
\frac{c_2}{n^{\alpha}}.
\]
\end{enumerate}
\end{thm}

\begin{proof}
By the self-similarity of Brownian motion (Lemma \ref{scale}), defining 
\[
	W^{(n)}(t) = \sqrt{n}\, W\left(\frac{t}{n}\right)
\]
for each $n \ge 1$, $W^{(n)}$ is standard Brownian motion as well. Use $W^{(n)}$ in the KMT construction described above to obtain $X^{(n)}$. Then by (\ref{eq:KMT}),
\begin{equation}
\Pb \max_{0 \le k \le n} \left|W\left(\frac{k}{n}\right) -
\ell_{X^{(n)}}\left(\frac{k}{n}\right)\right|
> 2 C_0 \frac{\log n}{\sqrt{n}} \rb < K_0 n^{-\alpha} \label{eq:KMT1}.
\end{equation}
By Cs{\"o}rg{\H{o}} and R{\'e}v{\'e}sz \cite{CRbook}*{Lemma 1.1.1} we have that for any $\epsilon > 0$ there exists $C>0$ such that for any
$n \ge 1$ and $c\ge((\alpha+1)(2+\eps))^{1/2}$, one has
\begin{equation} \label{eq:CSR}
\Pb \sup_{0 \le k < n, 0 \le t - \frac{k}{n} \le \frac{1}{n}} \left|W(t) - W\left(\frac{k}{n}\right)
\right| \ge c \left(\frac{\log n}{n} \right)^{\frac12} \rb \le C n^{1-\frac{c^{2}}{2+\epsilon }} \le C
n^{-\alpha }.
\end{equation}
Finally, by its definition,
\begin{equation} \label{eq:lXn}
\sup_{0 \le k < n, 0 \le t - \frac{k}{n} \le \frac{1}{n}} \left|\ell_{X^{(n)}}\left(t\right) -
\ell_{X^{(n)}}\left(\frac{k}{n}\right) \right| \le \frac{1}{\sqrt{n}}.
\end{equation}

(\ref{eq:KMT1}), (\ref{eq:CSR}), and (\ref{eq:lXn}) together prove the theorem.
\end{proof}

\begin{cor}\label{gaussian}
We can furthermore assert that for each $N\ge 1$,
\[
	\P\left(\bigcup_{n>N}\left\{ d(W,\ell_{X^{(n)}}) \ge  c\frac{\log n}{\sqrt{n}}\right\}\right)\le  \frac{c_2}{N}.
\]
\end{cor}
\begin{proof}
By Theorem \ref{Szab},
\[
	\P\left( \bigcup_{n>N}\left\{d(W,\ell_{X^{(n)}}) \ge c \frac{\log n}{\sqrt{n}}\right\}\right)\le \sum_{n>N} \frac{c_2}{n^\alpha} \le \frac{c_2}{N}.
\]
\end{proof}

\section{Application to Kolmogorov complexity}

The following lemma is a variation on a well-known fact about Kolmogorov complexity.
\begin{lem}\label{Kobe}
If $\P$ is a distribution on sequences 
\[
	\{x_n\}_{n\in\N}\in\prod_{n\in\N} \{\pm 1\}^n = \{\pm 1\}^{0}\times \{\pm 1\}^{1}\times\cdots
\]
such that the marginal distribution of each $x_n$ is uniform on $\{\pm 1\}^n$, then
\[
	\P(\exists b\,\,\forall n\,\,K(x_n)\ge n-b)=1.
\]
\end{lem}

\begin{proof}
Using $\sum_{\text{all }\sigma} 2^{-K(\sigma)} \le \sum_{p\text{ halts}} 2^{-|p|} \le 1$, we have
\begin{alignat*}{1}
	\P(\exists n\,\,K(x_n)<n-b) & \le \sum_{\{\sigma:K(\sigma)<|\sigma|-b\}} \P(x_{|\sigma|}=\sigma)=\sum_{\{\sigma:K(\sigma)<|\sigma|-b\}} 2^{-|\sigma|} \\
	& \le \sum_{\text{all }\sigma} 2^{-(K(\sigma)+b)} \le 2^{-b}.
\end{alignat*}
\end{proof}

We are now ready to show that Asarin's Theorem holds with the bound $n^{-1/10}$ replaced by $O(\,{\log n}/{\sqrt{n}})$.

\begin{thm}\label{AsarinImprov}
$W\in C[0,1]$ is a Martin-L\"of random Brownian motion if and only if there are constants $b$, $c$ such that for all but finitely many $n\in\N$ there is a string $x=(x_1,\ldots,x_n)\in\{\pm 1\}^n$ such that 
\[
	d(\ell_x,W)\le \frac{c \log n}{\sqrt{n}},\quad\text{and}\quad K(x_1,\ldots,x_n)\ge n-b.
\]
\end{thm}
\begin{proof}
In light of Asarin's Theorem, Theorem \ref{Asar}, it suffices to show the \emph{only if} direction.
For $b\in\N$, let
\[
	U_b=\bigcup_{n> b} \,\bigcap_{x\in\{\pm 1\}^n}\left\{W: \left(d(\ell_x,W)\le  \frac{c \log n}{\sqrt{n}}\right) \to (K(x)< n-b)\right\}.
\]
We verify that $U_{b}$ is a $\Sigma^0_1$ class: 
\begin{enumerate}
\item[(i)] $d(\ell_x,W)>c\log n/\sqrt{n}$ holds if and only if $W$ belongs to some $T_n$ such that all the members $B$ of $T_n$ satisfy $d(\ell_x,B)>c\log n/\sqrt{n}$, and 
\item[(ii)] the property $K(x)<n-b$ asserts that a short description of $x$ exists.
\end{enumerate}
Let $X^{(n)}$ be a random variable as in Theorem \ref{Szab}, and let
\[
	V_b=\bigcup_{n> b}  \left\{W: \left(d(\ell_{X^{(n)}},W)\le  \frac{c \log n}{\sqrt{n}}\right) \to (K(X^{(n)})< n-b)\right\}. 
\]
Note that $V_b$ is a random set, i.e. it is itself a random variable, but we will only use the following two auxiliary properties of $V_b$: $V_b$ always contains the deterministic set $U_b$, and $V_b$ has small Wiener measure. Indeed, by Corollary \ref{gaussian},
\[
	\mathbb P \left( \bigcup_{n> b} \left\{W: d(\ell_{X^{(n)}},W)>  \frac{c \log n}{\sqrt{n}}\right\}\right)\le \frac{c_2}{b},
\]
and by Lemma \ref{Kobe},
\[
	\P((\exists n)\, K(X^{(n)})< n-b)\le 2^{-b},
\]
so 
\[
	\P (U_b)\le \P (V_b) \le \frac{c_2}{b}+2^{-b},
\]
since clearly $U_b\subseteq V_b$.
Thus if $W$ is Martin-L\"of random, then
there is some $b$ such that for all $n>b$, there is some $x=(x_1,\ldots,x_n)$ with $\ell_x(t)$ that lies within $\frac{c\log n}{\sqrt{n}}$ of $W(t)$ and $K(x)\ge n-b$.
 \end{proof}
 
\begin{rem}
A complementary result to Theorem \ref{AsarinImprov} was obtained by Fouch\'e \cite{F1}*{Theorem 5}, who showed that from the first $n$ bits of a Martin-L\"of random real $A$ one can uniformly compute a finite linear combination of piecewise linear functions (of a different type from the ones considered here) that lies within $O(\log n/\sqrt{n})$ of an associated Martin-L\"of random path of Brownian motion $W_{A}$.  
\end{rem}
 
\section{A limitation on further improvements}
The rate $O(\log n/\sqrt{n})$ in our improved version of Asarin's Theorem, Theorem~\ref{AsarinImprov}, cannot be further improved to $\frac{1}{2}\sqrt{\frac{\log n}{n}}$. We now prove this by using L\'evy's analysis of the modulus of continuity of Brownian motion. A function $g:[0,1]\to\R$ is \emph{H\"older continuous of order $\gamma$} if there is a constant $C$ such that for all $x,y\in [0,1]$, $|f(x)-f(y)|\le C |x-y|^\gamma$. Wiener \cite{Wiener} showed that Brownian motion is almost surely H\"older continuous of any order $\gamma<1/2$, but this does not extend to order $\gamma=1/2$ (see for example Durrett \cite{Durrett}*{Exercise 2.4, p. 382}), and L\'evy obtained even more precise information.  

\begin{df}[L\'evy's modulus of continuity]
For a path of Brownian motion $W$, let \[
	\osc(\delta)=\sup\{|W(s)-W(t)| : s,t\in [0,1], |t-s|<\delta\}.
\]
\end{df}
\begin{thm}[L\'evy \cite{Levy}; see Durrett \cite{Durrett}*{p. 394}]\label{L}
Almost surely,
\[
	\limsup_{\delta\to 0} \osc(\delta)/(\delta\log(1/\delta))^{1/2}=\sqrt{2}.
\]
\end{thm}

Thus, we will give a lower bound for possible rates in Asarin's Theorem by using the fact that a
typical path of Brownian motion should have an increment $|W(s) - W(t)| \approx n^{-1/2}\sqrt{\log n}$
somewhere in an interval of length $1/n$, while an increment $|\ell_{x^{(n)}}(s) - \ell_{x^{(n)}}(t)|$
of an approximating broken line is only $ n^{-1/2}$ on the same interval.

\begin{cor}\label{LL}
Almost surely, for each $\theta>0$ and $\eps>0$ there exists $0<\delta<\eps$ and $s,t\in [0,1]$ with $|t-s|<\delta$ such that 
\[
	|W(s)-W(t)|\ge (\sqrt{2}-\theta)(\delta\log(1/\delta))^{1/2}.
\]
\end{cor}
\begin{proof}
By Theorem \ref{L},
\[
	\lim_{\eps\to 0}\sup_{0<\delta<\eps} \osc(\delta)/(\delta\log(1/\delta))^{1/2}\ge \sqrt{2},
\]
so for each $\eps>0$ and $\theta>0$ there exists $0<\delta<\eps$ such that 
\[
	 \osc(\delta)/(\delta\log(1/\delta))^{1/2}\ge \sqrt{2}-\theta.
\]
Thus there exists $s,t\in [0,1]$ with $|t-s|<\delta$ and
\[
	|W(s)-W(t)|\ge (\sqrt{2}-\theta)(\delta\log(1/\delta))^{1/2}.
\]
\end{proof}

\begin{cor}\label{ofcourse}
Almost surely for Brownian motion $W$, there is no $n_{W}$ such that for all $n\ge n_{W}$ there is an $x\in\{\pm 1\}^{n}$ with 
\[
	d(\ell_{x},W)\le \frac{1}{2}\sqrt{\frac{\log n}{n}}.
\]
\end{cor}
\begin{proof}
 Suppose otherwise; i.e., with positive probability there is an $n_{W}$ such that for all $n\ge n_{W}$, there is a length $n$ walk $g$ within $f(n):=\frac{1}{2}\sqrt{\frac{\log n}{n}}$ of $W$. 
 
Let $0<\gamma<1/2$. Since $\gamma>0$, there is a $\theta>0$ such that
\[
	2-\gamma \le \left(\frac{\sqrt{2}-\theta}{1+\theta}\right)^{2}.
\]
Let $\eps_{W}>0$ be small enough to guarantee that if $0<\delta\le\delta'<\eps_{W}$ and $\frac{1}{n+1}<\delta$, then
\begin{enumerate}
\item $(\delta\log(1/\delta))^{1/2}\le (\delta'\log(1/\delta'))^{1/2}$,

\item$n\ge n_{W}$,
\item $1\le \theta\sqrt{\log n}$, and
\item $n\ge 2$.
\end{enumerate}

 By Corollary \ref{LL}, for almost all $W$ we have that for all $\eps>0$, in particular for $\eps=\eps_{W}$ when $\eps_{W}$ exists, there exist $0<\delta<\eps$ and $s,t\in [0,1]$ so that 
\[
	|s-t| < \delta \quad\text{and}\quad
	(\sqrt{2}-\theta)\sqrt{\delta\log(1/\delta)}\le |W(t)-W(s)|.
\]
Let $n$ be such that $\frac{1}{n+1}<\delta<\frac{1}{n}$. Then for such $s$, $t$, using (1), (2), and (3), 
\[
	(\sqrt{2}-\theta)\sqrt{\frac{\log(n+1)}{n+1}}\le
	(\sqrt{2}-\theta)\sqrt{\delta\log(1/\delta)} \le |W(t)-W(s)|
\]
\[
	\le |W(t)-g(t)|+|g(t)-g(s)|+|g(s)-W(s)|
	\le 2f(n) + n^{-1/2}\le (1+\theta)\sqrt{\frac{\log n}{n}},
\]
and hence using (4),
\[
	2-\gamma\le
	\left(\frac{\sqrt{2}-\theta}{1+\theta}\right)^{2}
	\le
	\frac{\log n}{\log(n+1)}\cdot\frac{n+1}{n}\le \frac{3}{2},
\]
which contradicts $\gamma<1/2$. 
\end{proof}

We deduce that Asarin's Theorem, Theorem \ref{Asar}, does not hold for the bound $\frac{1}{2}\sqrt{\frac{\log n}{n}}$.

\begin{cor} \label{optimum}
The assumption that $W\in C[0,1]$ is a Martin-L\"of random Brownian motion does \textbf{not} imply that for all but finitely many $n\in\N$ there is any string $x=(x_1,\ldots,x_n)\in\{\pm 1\}^n$ such that 
\[
	d(\ell_x,W)\le \frac{1}{2}\sqrt{\frac{\log n}{n}}
\]
at all (let alone such a string with $K(x_1,\ldots,x_n)\ge n-c$).
\end{cor}
\begin{proof}
The set of Martin-L\"of random Brownian paths has positive measure (in fact measure 1), and so we would have a contradiction to Corollary \ref{ofcourse}.
\end{proof}

\section{Future work}\label{Future}
\subsection{Further narrowing of the bounds.} 
\begin{que}\label{mainq}
Does Asarin's Theorem hold for a rate $r(n)=o(\log n/\sqrt{n})$?
\end{que}

If the answer to Question \ref{mainq} is \emph{yes}, then the ball of radius $r(n)$ around $W$ is ``typical'' enough to contain $\ell_{x}$ for a string $x\in\{\pm 1\}^{n}$ with $K(x)\ge^{+}n$. If we then consider the random variable $X$ that selects (uniformly) a string $x\in \{\pm 1\}^{n}$ with $d(W,\ell_{x})\le r(n)$, we might suspect that $X$ would be (at least approximately) uniformly distributed on $\{\pm 1\}^{n}$ (for all but finitely many $n$). That $X$ could not literally be uniformly distributed is a consequence of the following result.

\begin{thm} \label{232}
Let $X_{k,n}$ $(k=1, \dots , n; n=1,2, \dots)$ be a triangular array of identically distributed random variables with zero expectation and unit variance such that $X_{1,n}, \dots , X_{n,n}$ are mutually independent. Let $S^{(n)}(k)=X_{1,n}+ \dots +X_{k,n}$ and $W^{(n)}(t)$ $(t \ge 0; n=1,2, \dots)$ be a sequence of Brownian motions such that
$$\max_{1 \le k \le n} \left|S^{(n)}(k) - W^{(n)}(k)\right| = o(\log n) \qquad \text{(a.s.)}. $$
Then the distribution of $X_{k,n}$ is standard normal, i.e., normal with mean 0 and variance 1.
\end{thm}
\begin{proof}
This theorem follows easily from Cs{\"o}rg{\H{o}} and R{\'e}v{\'e}sz \cite{CRbook}*{Theorem 2.3.2} if the i.i.d. sequence $X_n$ there is replaced by a triangular array $X_{k,n}$. (We mention that the theorem of Cs{\"o}rg{\H{o}} and R{\'e}v{\'e}sz \cite{CRbook}*{Theorem 2.3.2} was strengthened by Bass and Burdzy \cite{BassBurdzy}*{Theorem 5.6}.) The underlying Theorems 2.3.1, 2.4.3-2.4.5 there can be modified accordingly as well.
\end{proof}

\subsection{Rate distortion theory}

We are grateful to the referee for the following interesting problem. \begin{que}\label{referee}
What is the complexity of the \emph{simplest} walk within $O(\log n/\sqrt{n})$ of the Brownian path? 
\end{que}

In the framework of classical rate distortion theory, this asks how low the complexity of a walk can be if we still want to be able to reconstruct a reasonable approximation to the Brownian path from it. Question \ref{referee} appears to be related to the question of \emph{how many} walks are close to the Brownian path, i.e., the size of balls 
\[
	B_{\eps}(W)=\{x\in\{\pm 1\}^{n}: d(\ell_{x},W)\le\eps\}
\]
centered at the Brownian path. Chen \cite{Chen}*{Theorem 2} obtained precise information about the size of balls centered at the constant function $0$. His work is not immediately applicable here, however, because by Corollary \ref{ofcourse} there will be choices of $\eps$ where $B_{\eps}(W)=\nil$, while $B_{\eps}(0)$ is fairly large.

\subsection{Schnorr randomness} While the most studied form of algorithmic randomness is Martin-L\"of randomness, there are other variants such as Schnorr randomness. Here tests are required to have the probability of $U_n$ equal to $2^{-n}$, or equivalently any computable function of $n$ that goes effectively to zero. Schnorr randomness is preferable in the sense that if a finite number of almost sure properties is shown to hold for each Schnorr random function, then there is in fact a computable function displaying this almost sure behavior. It would be somewhat disturbing to prove a law without being able to provide a computable example illustrating it. Fortunately, most laws encountered in practice hold for all Schnorr random functions.  

\begin{que}
Is there an analogous result to our main result, Theorem \ref{AsarinImprov}, for Schnorr random Brownian motion?
\end{que}

In one direction, the problem here is to compute the measure of analogues of the sets $U_n$. We suspect that this might be possible by analyzing the speed of convergence in Donsker's theorem.

\end{document}